\documentclass[12pt]{amsart}

\usepackage{amsthm}
\usepackage{amsmath}
\usepackage{amssymb}
\usepackage{hyperref}
\usepackage{mathrsfs}
\usepackage{array}

\theoremstyle{definition}
\newtheorem{thm}{Theorem}[section]

\newtheorem*{prop*}{Proposition}

\newtheorem{rmk}[thm]{Remark}

\newcommand{\R}{\mathbb{R}}
\newcommand{\C}{\mathbb{C}}
\newcommand{\Z}{\mathbb{Z}}

\newcommand{\om}{\omega}
\newcommand{\la}{\lambda}

\newcommand{\eps}{\varepsilon}

\newcommand{\Ham}{\mrm{Ham}}

\newcommand{\id}{\mrm{id}}

\def\mrm#1{{\mathrm{#1}}}

\renewcommand{\leq}{\leqslant}
\renewcommand{\geq}{\geqslant}

\newcommand{\esemail}{{egor.shelukhin@umontreal.ca}}
\newcommand{\fbemail}{{filipvbrocic@gmail.com}}

\usepackage{filecontents}

\begin{document}

\title{A counterexample to {L}agrangian {P}oincar\'e recurrence}

\author{Filip Bro\'ci\'c$^1$ and Egor Shelukhin$^2$}

\footnotetext[1]{Lehrstuhl für Analysis und Geometrie, Universität Augsburg, Universitätsstrasse 14, 86159 Augsburg, Germany,\vspace{0.2cm} \fbemail}
\footnotetext[2]{D\'epartement de Math\'ematiques et de Statistique, Universit\'e de Montr\'eal, C.P. 6128 Succ.  Centre-Ville Montreal, QC H3C 3J7, Canada, \esemail}




\begin{abstract}
We provide a counterexample to the Lagrangian Poincar\'e recurrence conjecture of Ginzburg and Viterbo in all dimensions $6$ and greater.\end{abstract}

\maketitle

\section{Introduction}

The well-known Lagrangian Poincar\'e recurrence conjecture (LPR) in symplectic topology and Hamiltonian dynamics states that for every closed manifold $(M,\om),$ closed Lagrangian submanifold $L$ of $M,$ and Hamiltonian diffeomorphism $\varphi \in \Ham(M,\om)$ of $M,$ there should exist an integer sequence $k_i \to \infty$ such that $\varphi^{k_i}(L) \cap L \neq \emptyset$ for all $i.$ We refer to \cite{GG-ai} for a detailed discussion of this conjecture put forth independently by Ginzburg and by Viterbo. It is related to an idea of Hofer and Polterovich that certain subsets of symplectic manifolds should be superrecurrent under Hamiltonian mappings: more recurrent than suggested by ergodic theory. In dimension two, LPR holds immediately for considerations of area, and is indeed a weaker version of the Poincar\'e recurrence theorem. In dimension $4,$ this conjecture was established in \cite{PSh} for certain product Lagrangian tori $L$ inside $S^2 \times S^2$ with a split symplectic form $\om = A\sigma \oplus a\sigma,$ $\int_{S^2} \sigma = 1,$ for suitable areas $a< A$ and all $\varphi \in \Ham(M,\om).$ It was also established for every closed Lagrangian $L \subset \C P^n$ with respect to Hamiltonian pseudo-rotations $\varphi \in \Ham(\C P^n)$ of projective spaces in \cite{GG-invent} (see also \cite{JS}). In this paper, we disprove the LPR conjecture for a class of irrational Lagrangian tori in every symplectic manifold of dimension $6$ and greater. Let $\Ham_c(M,\om)$ denote the group of compactly supported Hamiltonian diffeomorphisms of $(M,\om).$

\begin{thm}\label{thm: main}
For every symplectic manifold $M$ of dimension at least $6$, there exists a closed Lagrangian submanifold $L \subset M$ and a Hamiltonian diffeomorphism $\varphi \in \Ham_c(M,\omega)$ such that $$\varphi^k(L) \cap L = \emptyset  \text{ for all } k \geq 1.$$
\end{thm}

It appears therefore that in higher dimensions, if the Lagrangian Poincar\'e recurrence conjecture may hold, then it should hold for the class of Bohr-Sommerfeld Lagrangians, or at least rational ones in the sense that $\om$ has, up to rescaling, integral periods on disks in $\pi_2(M,L).$ Forthcoming work \cite{BPS} shall make new progress in this direction. Finally, in dimension $4,$ techniques of \cite{PSh}, related to Lagrangian Floer theory in orbifolds, should be applicable in greater generality. We expect these two directions of research to complement Theorem \ref{thm: main} by providing positive results regarding LPR.

Our construction relies on the work \cite{Chek, CS-tori} brought to our attention by \cite{CL24, R24}, where \cite[Theorem 1.1(ii)]{CS-tori} was used to prove the existence of Hamiltonian diffeomorphisms bringing a given Lagrangian $L$  to an arbitrarily small Weinstein neighborhood of $L$ without intersecting $L$ and hence not being exact in this neighborhood. 

The idea of the proof of Theorem \ref{thm: main} is that of creating ``linear chaos" related to a hyperbolic matrix in $GL(2,\Z)$ ``represented" by a Hamiltonian diffeomorphism. We now outline our approach in the case of the standard symplectic $\R^{2n}.$ Following Chekanov \cite{Chek}, we exhibit a family of Lagrangian tori $L_w$ in $\R^{2n}$ parametrized by $w$ in a disk $\Delta'$ around the origin in $\R^2$ and a Hamiltonian diffeomorphism $\varphi \in \Ham_c(\R^{2n})$ such that for all $w$ in a small sub-disk $\Delta$ in $\Delta',$ $\varphi(L_w) = L_{Bw}$ where $B \in GL(2,\Z)$ is a hyperbolic matrix. Picking a contracting eigenvector $e \in \Delta,$ our desired example is $L=L_e$ as by induction $\varphi^k(L_e) = L_{\la^k e},$ $B^k e =\la^k e \in \Delta,$ for all $k \geq 0$ and these Lagrangian submanifolds are pairwise disjoint.

\begin{rmk}
Theorem \ref{thm: main} also provides an example of a Lagrangian submanifold $L$ with an infinite packing number: namely the existence of Hamiltonian diffeomorphisms $\varphi_k \in \Ham_c(M,\om),$ $k\geq 1,$ such that $\varphi_k(L)$ are all pairwise disjoint. Indeed, one need only take $\varphi_k = \varphi^k$ for $\varphi \in \Ham_c(M,\om)$ provided by the theorem. Existence of such examples in dimension at least $6$ is also an immediate consequence of \cite[Theorem 1.1(ii)]{CS-tori} (see \cite[Section 1.3]{B}). However, it seems generally easier to produce infinitely packable Lagrangians than counterexamples to LPR. For instance, as communicated to us by Leonid Polterovich, infinitely packable Lagrangians in dimension $4$ were constructed by Brendel and Kim \cite{B, BK}. However, at the time of writing, we were not able to obtain counterexamples to LPR in dimension $4.$
\end{rmk}

\section*{Acknowledgments}
This project was strongly influenced by the 2024 Workshop in Symplectic Topology in Belgrade. We thank its organizers for a wonderful event and R\'emi Leclercq for an inspiring lecture which has introduced us to \cite{CS-tori}. We also thank Leonid Polterovich for introducing us to Lagrangian Poincar\'e recurrence and for numerous helpful comments and discussions. Finally, hearty thanks to Felix Schlenk for finding a mistake in the previous version of this paper and for helpful comments on the exposition. ES was partially supported by the NSERC, Fondation Courtois, and the Alfred P. Sloan foundation. While at the Universit\'e de Montr\'eal, FB was supported by the ISM scholarship and the Fondation Courtois. 

\section{Proof of the main result}

\begin{proof}[Proof of Theorem \ref{thm: main}]
First, it is sufficient to prove the theorem for the standard symplectic vector space $(\R^{2n},\om_{st})$ of dimension $2n \geq 6.$ Indeed, once $L'' \subset \R^{2n}$ and $\varphi'' \in \Ham_c(\R^{2n},\om_{st})$ satisfying the theorem are found, then $L' = r_{s}(L'')$ and $\varphi' = r_{s} \circ \varphi'' \circ r_{s}^{-1}$ can be arranged to be contained, and, respectively, supported in an arbitrarily small neighborhood of the origin for a rescaling $r_{s}: \R^{2n} \to  \R^{2n},$ $x \mapsto s \cdot x$ by a sufficiently small $s > 0.$ 
Now let $B^{2n}(C) \subset \R^{2n}$ be the ball of capacity $C>0.$ Given an arbitrary symplectic manifold $(M,\om)$, and a Darboux chart $D$ in $M$ with a symplectomorphism $\iota: B^{2n}(\delta) \to D,$ we may assume that $L' \subset B^{2n}(\delta)$ and $\varphi'$ is supported in $B^{2n}(\delta).$ Hence the Lagrangian submanifold $L = \iota(L') \subset M$ and $\varphi \in \Ham_c(M,\om)$ given by $\iota \circ \varphi' \circ \iota^{-1}$ on $D$ and by $\id$ outside $D$ satisfy the theorem. 

For a matrix $R \in GL(p,\Z),$ $p \geq 1,$ let $R^*: T^* T^p \to T^* T^p$ denote the symplectomorphism induced by $(R^{-1})^* \times R: T^*\R^p \to T^*\R^p.$ Now, Chekanov \cite[Theorem 3.8]{Chek} constructs a Lagrangian torus $T \subset \R^{2n}$ for $n=k+m,$ $k, m \geq 1$ such that for every $A \in GL(k,\Z)$ there exist: a) neighborhoods $V \subset V'$ of $T$ with a symplectomorphism $I': W_{\eps'} \to V'$ which restricts to $I: W_{\eps} \to V$ for $0<\eps \ll \eps'$, where we set $W_{\sigma} = D^*_{\sigma} T^k \times D^*_{\sigma} T^m \subset T^*T^n,$ for all $\sigma >0$ (we think of $V$ as being very small) and b) a Hamiltonian diffeomorphism $\varphi \in \Ham_c(\R^{2n})$ with the property that $\varphi(V) \subset V'$ and on $V,$  $\varphi = I' \circ (A \times \id)^* \circ I^{-1}.$ Indeed, in the notations of \cite[Proof of Theorem 3.8]{Chek}, it is enough to take a Lagrangian $m$-torus $L \subset T^*B$ sufficiently close to the zero section to achieve this.

Set $k=2$ and $m=n-2 \geq 1,$ $A \in GL(2,\Z)$ given by $A=\begin{bmatrix}
1 & 1\\
1 & 0
\end{bmatrix},$ $B = (A^{-1})^* = \begin{bmatrix}
0 & 1\\
 1 & -1
\end{bmatrix}.$ For $w \in \R^2$ consider the Lagrangian torus $T_w = \{w\} \times T^2 \times \{0\} \times T^m \subset T^* T^2 \times T^*T^m = T^*T^n$ and whenever $T_w \subset W_{\eps'}$ set $L_w = I'(T_w).$ Now let $e = \kappa (1,\la)$ be an eigenvector of $B$ with eigenvalue $\la = \frac{-1 + \sqrt{5}}{2} \in (0,1).$ Here $\kappa > 0$ is chosen to be sufficiently small so that $\{e\} \times T^2 \times \{0\} \times T^m \subset W_{\eps}.$ Consider the Lagrangian torus $L_e = I(\{e\} \times T^2 \times \{0\} \times T^m).$ As $L_e \subset V,$ by construction $\varphi(L_e) = L_{B(e)} = L_{\la e} \subset V.$ By induction, $\varphi^k(L_e) = L_{\la^k e} \subset V$ for all $k \geq 0.$ In particular setting $L=L_e,$ $\varphi^k(L)$ are all pair-wise disjoint. This finishes the proof.\end{proof}

\begin{rmk}
By \cite[Proposition 3.4]{Chek} the tori $L_{\la^k e},$ $k\geq 0,$ produced in our argument are Hamiltonian isotopic to product tori. Recall that a product torus is $T(a_1,\ldots, a_n) = \{ \pi |z_j|^2 = a_j,\; 1\leq j \leq n\} \subset \C^n,$ for certain $a_j > 0,$ $1 \leq j \leq n.$ For example, we may choose the parameters in Chekanov's construction so that for all $k \geq 0,$ $L_{\la^k e}$ is Hamiltonian isotopic to $T_k = T(b+\kappa\la^k, b+\kappa \la^{k+1}, b, \ldots, b).$ It is interesting to note that there does not exist a Hamiltonian isotopy that works for all the $L_{\la^k e}$ simultaneously, as otherwise $T = T(b,\ldots, b),$ which is the Hausdorff limit of the $T_k,$ would be Hamiltonian isotopic to $L_0,$ which the Hausdorff limit of the $L_{\la^k e}.$ However, $T$ and $L_0$ are not Hamiltonian isotopic due to \cite[Theorem 4.2]{Chek}.
\end{rmk}

\vspace{-8pt}
\bibliography{refs}
\bibliographystyle{alpha}

\end{document}